\documentclass{amsart}[11pt, a4paper, headings=standardclasses, dvipsnames, bibliography=totocnumbered]

\usepackage[utf8]{inputenc}
\usepackage[english]{babel}

\usepackage[style=alphabetic, backend=biber, maxalphanames=4, maxbibnames=4]{biblatex}
\usepackage{hyperref}

\usepackage{import}

\usepackage{tikz-cd}
\usepackage{xcolor}
\usetikzlibrary{arrows,decorations.pathreplacing,calligraphy,patterns,bending,calc}

\definecolor{Green}{HTML}{008000}
\definecolor{BrightGreen}{HTML}{32CD32}

\usepackage{wrapfig}
\usepackage{subcaption}
\usepackage{xfrac}
\usepackage[makeroom]{cancel}
\usepackage{array}
\usepackage{tabularx}
\usepackage{booktabs}
\usepackage{float}
\usepackage{faktor}

\usepackage[T1]{fontenc}
\usepackage{hyperref}
\usepackage{amsfonts}
\usepackage{graphicx}
\usepackage{amsthm}
\usepackage{amssymb}
\usepackage{amsmath}
\usepackage{mathtools}
\usepackage{hyperref}

\usepackage{tikz}

\newcommand{\devil}{{\color{epurple} \begin{tikzpicture}[baseline=-0.5ex,scale=0.15] \fill (0,0) circle (1); \fill (-0.5,0.8) -- (-0.8,1.8) -- (-0.1,1.1) -- cycle; \fill (0.5,0.8) -- (0.8,1.8) -- (0.1,1.1) -- cycle; \fill[white] (-0.3,0.2) circle (0.12); \fill[white] (0.3,0.2) circle (0.12); \draw[white,thick] (-0.4,-0.3) arc (200:340:0.4); \end{tikzpicture}} }

\newcommand{\C}{\mathbb{C}} 
\newcommand{\U}{U_q(\mathfrak{sl}_2)}
\newcommand{\mU}{\mathbf{1}}
\DeclareMathOperator{\Hom}{Hom}
\newcommand{\id}{\mathrm{id}}

\newtheorem{theo}{Theorem}
\numberwithin{theo}{section}
\newtheorem{lemm}[theo]{Lemma}

\theoremstyle{definition}
\newtheorem{defi}[theo]{Definition}

\newtheorem{rema}[theo]{Remark}
\newtheorem{obse}[theo]{Observation}

\newtheorem{fact}[theo]{Fact}
\newtheorem{nota}[theo]{Notation}

\colorlet{ered}{red!60!black}
\colorlet{ered}{red!60!black}
\colorlet{eblue}{blue!60!black}
\colorlet{egreen}{green!60!black}
\colorlet{epurple}{purple!60!black}
\definecolor{gold}{RGB}{255,215,0}
\colorlet{egold}{epurple}

\newcommand{\e}{\textcolor{ered}{E}}
\newcommand{\f}{\textcolor{eblue}{F}}
\renewcommand{\k}{\textcolor{egreen}{K}}

\newcommand{\Rep}{\mathbf{Rep^{fd}}(\U)}

\usepackage{pgfplots}
\usetikzlibrary{arrows}

\newcommand*{\blank}{{-}}
\newcommand*{\bblank}{{=}}

\usetikzlibrary{calc}
\usetikzlibrary{tqft}
\newcommand{\cbox}[1]{ \vcenter{ \hbox{#1} } }
\usetikzlibrary{patterns,decorations.pathreplacing}
\tikzset{
	tldiagram/.style={thick, scale=0.35}
}
\newcommand{\tlcoord}[2]
{
	(2*#2 , 3*#1)
}

\newcommand{\lineup}{-- ++(0,3)}
\newcommand{\halflineup}{-- ++(0,1.5)}

\newcommand{\linewave}[2]{
	.. controls +(0,1.5*#1) and +(0,1.5*-#1) .. ++(2*#2, 3*#1)
}

\newcommand{\capright}{arc (180:0:1)}
\newcommand{\capleft}{arc (0:180:1)}

\newcommand{\cupleft}{arc (360:180:1)}

\newcommand{\dcaprightsmall}{arc (180:90:1) node {\scalebox{0.7}{\textbullet}} arc (90:0:1)}

\newcommand{\dcuprightsmall}{arc (180:270:1) node {\scalebox{0.7}{\textbullet}} arc (270:360:1)}

\newcommand{\negtwista}[2]{
	\draw [ultra thick] \tlcoord{0+3*#1}{0+2*#2} -- \tlcoord{0.45+3*#1}{0.4+2*#2};
	\draw [ultra thick] \tlcoord{0.55+3*#1}{0.65+2*#2} -- \tlcoord{1+3*#1}{1+2*#2};
	\draw  \tlcoord{0+3*#1}{1+2*#2} -- \tlcoord{1+3*#1}{0+2*#2};
}

\newcommand{\negtwist}[2]{
	\negtwista{#1}{#2}
	\negtwist{1+#1}{#2}
}

\newcommand{\predtlcupcap} 
{
	\begin{tikzpicture}[scale=0.092]
		\draw \tlcoord{1}{0} \dcuprightsmall;
		\draw \tlcoord{0}{0} \dcaprightsmall;
	\end{tikzpicture}
}

\newcommand{\posredbluecrossing}[2]{
	\draw[black] \tlcoord{#1}{#2} \linewave{1}{1};
	\draw[white, double=epurple] \tlcoord{#1}{#2+2} \linewave{1}{-1} ;
}

\newcommand{\posblueredcrossing}[2]{
	\draw[black] \tlcoord{#1}{#2+2} \linewave{1}{-1} ;
	\draw[white, double=epurple] \tlcoord{#1}{#2} \linewave{1}{1};
}
\newcommand{\iso}{
	\cbox{
	\begin{tikzpicture}[tldiagram, scale=0.65]
		\posblueredcrossing{0}{0}
	\end{tikzpicture}
	}
}
\newcommand{\isolabel}{
	\cbox{
		\begin{tikzpicture}[tldiagram, scale=0.65]
			\posblueredcrossing{0}{0}
			\node (n1) at \tlcoord{-0.25}{0} {$S$};
			\node (n2) at \tlcoord{-0.25}{1} {$V$};
			\node (n3) at \tlcoord{1.25}{0} {$V$};
			\node (n3) at \tlcoord{1.25}{1} {$S$};
		\end{tikzpicture}
	}
}
\newcommand{\inv}{
	\cbox{
		\begin{tikzpicture}[tldiagram, scale=0.65]
			\posredbluecrossing{0}{0}
		\end{tikzpicture}
	}
}
\newcommand{\blackcap}{
	\cbox{
		\begin{tikzpicture}[tldiagram, scale=0.65]
			\draw[black] \tlcoord{0}{0} \capright;
		\end{tikzpicture}
	}
}
\newcommand{\blackcaplabel}{
	\cbox{
		\begin{tikzpicture}[tldiagram, scale=0.65]
			\draw[black] \tlcoord{0}{0} \capright;
			\node (n1) at \tlcoord{-0.25}{0} {$V$};
			\node (n2) at \tlcoord{-0.25}{1} {$V$};
			\node (n3) at \tlcoord{0.58}{0.5} {$\mU$};
		\end{tikzpicture}
	}
}

\newcommand{\purplecap}{
	\cbox{
		\begin{tikzpicture}[tldiagram, scale=0.65]
			\draw[epurple] \tlcoord{0}{0} \capright;
		\end{tikzpicture}
	}
}
\newcommand{\purplecaplabel}{
	\cbox{
		\begin{tikzpicture}[tldiagram, scale=0.65]
			\draw[epurple] \tlcoord{0}{0} \capright;
			\node (n1) at \tlcoord{-0.25}{0} {$S$};
			\node (n2) at \tlcoord{-0.25}{1} {$S$};
			\node (n3) at \tlcoord{0.58}{0.5} {$\mU$};
		\end{tikzpicture}
	}
}
\newcommand{\invlabel}{
	\cbox{
	\begin{tikzpicture}[tldiagram, scale=0.65]
		\posredbluecrossing{0}{0}
		\node (n1) at \tlcoord{-0.25}{0} {$V$};
		\node (n2) at \tlcoord{-0.25}{1} {$S$};
		\node (n3) at \tlcoord{1.25}{0} {$S$};
		\node (n3) at \tlcoord{1.25}{1} {$V$};
	\end{tikzpicture}
	}
}

\newcommand{\goldiso}{
	\cbox{
		\begin{tikzpicture}[tldiagram, scale=0.65]
			\draw[black, dashed] \tlcoord{0}{1} \linewave{1}{-1};
			\draw[white, double=epurple] \tlcoord{0}{0} \linewave{1}{1};
		\end{tikzpicture}
	}
}
\newcommand{\goldinv}{
	\cbox{
		\begin{tikzpicture}[tldiagram, scale=0.65]
			\draw[black, dashed] \tlcoord{0}{0} \linewave{1}{1};
			\draw[white, double=epurple] \tlcoord{0}{1} \linewave{1}{-1};
		\end{tikzpicture}
	}
}
				
\title{No braiding for quantum sl2 reps - a short proof}
\author{Liao Wang, Zbigniew Wojciechowski}
\begin{document}
	\begin{abstract}
		We give a short and self-contained proof that the category of all finite dimensional representations of quantum $\mathfrak{sl}_2$ does not admit any braided structure. 
		We also show that it admits a braided module category structure over its usual braided subcategory. This braided module structure arises from the full twists in the usual braid group.
	\end{abstract}
	\maketitle
	\section{Introduction}
	\begin{center}
		\textit{`As commonly known the category $\Rep$ of all finite dimensional representations of the quantum group $\U$ is braided monoidal.'}
	\end{center} 
	This is a common phrase at some conferences, which is then often (but not always) interrupted by one enlightened researcher pointing out that is not, and that one has to restrict to an appropriate subcategory. Sometimes other scientists start energetically nodding their head or saying things like `Yes, of course', while others are confused. The goal of this note is to provide a detailed, self-contained, but short proof of the fact, there is no possible braiding on $(\Rep, \otimes)$ due to an impish sign $\textcolor{epurple}{-1}$.
	
	This fact is certainly known to experts for at least 40 years, yet there is no standard reference in the literature. The closest references we could find are the following. One is Jantzen's lecture notes \cite[\S3.19]{jantzen96}, where it is discussed that the action of the quasi $R$-matrix cannot be properly rescaled. Our computation is more general, since we do not use $R$-matrices. The other one is \cite[\S4.1]{beliakovagorsky24} by Beliakova--Gorsky, where it is mentioned that a sign appears on weight spaces, when trying to find a braiding. Our approach is based on the diagrammatic presentation in \cite{stroppelwojciechowski26}. We want to point out that \cite{tingley17}, which discusses $\U$ and subtle sign issues focuses on a different minus sign related to the equivalence to the Temperley--Lieb category and not to type $-1$ representations. 

	Instead of the braided monoidal structure, we show that there is a \emph{braided module category} structure on $\Rep$ over the subcategory of type $1$ representations. Interestingly, this braided module structure comes from the \emph{full twist} in the braid group, which generates the center of the braid group on more than two strands. It would be interesting to explore other aspects and interpretations of this braided module structure.

	\subsection*{Acknowledgments}
	We thank everyone whom we had the above nitpicky conversation with, thus motivating us to write this paper. Specifially we would like to thank Daniel Bermudez, Istvan Heckenberger, Ulrich Krähmer, Stefan Kolb, Carlos Mendes Góngora, Isak Sundelius, Mateusz Stroiński, Catharina Stroppel, and Eduardo Venturini. We thank Paul Wedrich for the final remark in this paper.

	\subsection*{Tools and computational resource disclosure}
	We used AI to help with \LaTeX, specifially the \devil-emoji created by Microsoft Copilot. We asked Copilot and ChatGPT 5.6 for references proving or disproving that the category of all finite dimensional representations is braided monoidal. While Microsoft Copilot tried to convince us that the category is braided and that our paper is wrong, ChatGPT was aware of sign problems and that one has to take a subcategory. ChatGPT didn't know a specific reference proving that it is not braided monoidal, but it helped us find \cite[\S4.1]{beliakovagorsky24}. Furthermore it started proving it using that paper, even though we didn't ask for it. We completely ignored that part, since we have our own proof.
	
	\section{The argument}
	We work over a field $k$ of $\operatorname{char}(k)\neq 2$ and fix a unit $q\in k^{\times}$, for instance $k=\C(q)$, but all our arguments work for an arbitrary commutative ring $k$ with $1+1\neq 0$. 
	\begin{defi}
		The \textbf{quantum group} $\U$ is the $k$-algebra with generators $E$, $F$, $K$ and $K^{-1}$, subject to the following relations
		\[
			KE=q^2EK, \quad KF=q^{-2}FK, \quad EF-FE=\frac{K-K^{-1}}{q-q^{-1}}, \quad KK^{-1}=1=K^{-1}K.
		\]
	\end{defi}
	
	\begin{defi}
		The \textbf{(type $1$) standard representation} $V$ of $\U$ is the $2$-dimensional vector space with basis $x$ and $y$. 
		The \textbf{type $1$ trivial representation} $\mU$ of $\U$ is the $1$-dimensional vector space with basis $1$. 
		The \textbf{type $-1$ trivial representation} $S$ of $\U$ is the $1$-dimensional vector space with basis $\devil$. The action of $\U$ on these representations is depicted as follows
		\[
		V\coloneqq \, \begin{tikzcd}
			y \arrow[r, "1", bend left, ered] \arrow["q^{-1}"', loop, distance=2em, in=305, out=235, egreen] & x \arrow[l, "1", bend left, eblue] \arrow["q"', loop, distance=2em, in=305, out=235, egreen]
		\end{tikzcd} \quad , \qquad  \mU\coloneqq \, \begin{tikzcd}
		1 \arrow["1"', loop, distance=2em, in=305, out=235, egreen]
	\end{tikzcd}\quad , \qquad S\coloneqq \, \begin{tikzcd}
			\devil \arrow["-1"', loop, distance=2em, in=305, out=235, egreen]
		\end{tikzcd}.
		\]
		Here the action of $\k$ is depicted by a \textcolor{egreen}{green arrow}, $\f$ by a \textcolor{eblue}{blue arrow} and $\e$ by a \textcolor{ered}{red arrow}. The scalars give the structure constants of the action. For instance $\f\cdot x=\textcolor{eblue}{1}\cdot y=y$, while $\k\cdot x=\textcolor{egreen}{q} x$. Out-going arrows which are not drawn are regarded as that operator acting by zero, for instance $\e\cdot x=0$ and $\f\cdot \devil=0$. 
	\end{defi}
	We will use these `action graphs' as in the previous definition from now on.
	\begin{fact}
		We denote by $\otimes\coloneqq \otimes_{k}$ the tensor product of $k$-vector spaces. The comultiplication $\Delta\colon \U\to \U\otimes \U$ is the algebra homomorphism defined by
		\[
		\Delta(E)=E\otimes 1 + K\otimes E, \quad \Delta(F)=F\otimes K^{-1}+ 1\otimes F, \quad \Delta(K)=K\otimes K.
		\]
		It turns $\U$ into a bialgebra and even a Hopf algebra. The category $\Rep$ of all left finite dimensional $\U$-representations is equipped with the monoidal structure $\otimes$ coming from this comultiplication. The monoidal unit is the type $1$ trivial representation $\mU$.
	\end{fact}
	\begin{nota}
		We will use associativity isomorphisms in this paper implicitly in the notation. Thus we pretend like $\Rep,\otimes$ is a strict monoidal category, even though it is not.
	\end{nota}
	The following is the main result of this paper.
	\begin{theo}[Devils Sign] \label{theo no braiding}
		There is no natural isomorphism $c\colon \blank \otimes \bblank \to \bblank \otimes \blank$ turning $(\Rep,\otimes)$ into a braided monoidal category. 
	\end{theo}
	Before we prove the theorem, we need three lemmas, giving us certain morphisms and relations between them.
	\begin{lemm} \label{lemma V otimes S}
		We have
		$V\otimes S\cong S\otimes V$. A basis of  $\Hom_{\U}(V\otimes S, S\otimes V)$ 
		is given by the isomorphism
		\[
			\inv\coloneqq \invlabel\colon \quad x\otimes \devil \mapsto \devil \otimes x, \qquad y\otimes \devil \mapsto \textcolor{epurple}{(-1)} \cdot \devil \otimes y,
		\]
		defined on basis vectors. Its inverse is given by 
		\[
		\iso\coloneqq \isolabel\colon \quad \devil\otimes x \mapsto x\otimes \devil, \qquad \devil\otimes y \mapsto \textcolor{epurple}{(-1)} \cdot y \otimes \devil.
		\]
	\end{lemm}
	\begin{proof}
		As $\U$-representations $V\otimes S$ and $S\otimes V$ look as follows:
		\[
		V\otimes S\colon \, \begin{tikzcd}
			y\otimes \devil \arrow[r, "1", bend left, ered] \arrow["-q^{-1}"', loop, distance=2em, in=305, out=235, egreen] & x\otimes \devil \arrow[l, "-1", bend left, eblue] \arrow["-q"', loop, distance=2em, in=305, out=235, egreen]
		\end{tikzcd} \quad , \qquad S\otimes V\colon \, \begin{tikzcd}
			\devil \otimes y \arrow[r, "-1", bend left, ered] \arrow["-q^{-1}"', loop, distance=2em, in=305, out=235, egreen] & \devil \otimes x \arrow[l, "1", bend left, eblue] \arrow["-q"', loop, distance=2em, in=305, out=235, egreen]
		\end{tikzcd}
		\]
		This directly implies the lemma, since morphisms preserve $K$-eigenspaces.
	\end{proof}
	\begin{obse} \label{observation braiding candidate}
		Assume the the monoidal category $(\Rep, \otimes)$ is braided with braiding natural isomorphism $c\colon \blank \otimes \bblank \, \Rightarrow \, \bblank \otimes \blank$.
		Then by Lemma \ref{lemma V otimes S} we have
		\[
			c_{S, V}=\lambda \cdot \iso \text{ for some $\lambda\in k^{\times}$, and } c_{S, V}^{-1}=\frac{1}{\lambda} \cdot \inv.
		\]
		We will show that such $c_{S,V}$ cannot be natural. In other words we show there is no $\lambda$, such that this `braiding candidate' allows to slide morphisms through.
	\end{obse}
	The following lemma is a small part of \cite[Definition 4.1, Theorem 4.6]{stroppelwojciechowski26}, where the full list of relations for $\Rep$ is provided. Note that there the \textcolor{red}{red bullet} represents $V$ and the \textcolor{blue}{blue minus} represents $S$.
	\begin{lemm} \label{lemma with minus}
		There is a morphism $V\otimes V\to \mU$ given by 
		\[
			\blackcap\coloneqq \blackcaplabel \colon V\otimes V\to \mU,
		\]
		given on a basis by $x\otimes x \mapsto 0$, $x\otimes y\mapsto -q^{-1}$, $y\otimes x\mapsto 1$, $y\otimes y\mapsto 0$.
		Furthermore the following Brauer relation between morphisms $V\otimes S\otimes V\to S$ holds:
		\[
			\cbox{
				\begin{tikzpicture}[tldiagram, scale=0.65]
					\posblueredcrossing{0}{1}
					\draw[black] \tlcoord{0}{0} \lineup \capright;
					\draw[epurple] \tlcoord{1}{2} \halflineup;
				\end{tikzpicture}
			} = \textcolor{epurple}{(-1)} \cdot \cbox{
				\begin{tikzpicture}[tldiagram, scale=0.65]
					\posredbluecrossing{0}{0}
					\draw[black] \tlcoord{0}{2} \lineup \capleft;
					\draw[epurple] \tlcoord{1}{0} \halflineup;
				\end{tikzpicture}
			}.
		\]
	\end{lemm}
	\begin{proof}
		As $\U$-representation $V\otimes V$ looks as follows:
		\[
		V\otimes V\colon \quad \begin{tikzcd}[row sep=large]
			& x\otimes y \arrow[ld, "q"', bend left, eblue] \arrow[rd, "q", bend left,  ered], \arrow[loop, "1"', distance=2em, in=135, out=45, egreen, looseness=8]     &                                                                         \\
			y\otimes y \arrow["q^{-2}"', loop, distance=3em, in=235, out=135, egreen]  \arrow[ru, "1", bend left,  ered] \arrow[rd, "q^{-1}"', bend right, ered] &                                                                     & x\otimes x \arrow[lu, "1"', bend left, eblue] \arrow[ld, "q^{-1}", bend right,  eblue] \arrow["q^{2}"', loop, distance=3em, in=45, out=-45, egreen] \\
			& y\otimes x \arrow[lu, "1", bend right,  eblue] \arrow[ru, "1"', bend right, ered] \arrow[loop, "1"', distance=2em, in=-45, out=-135, egreen, looseness=8]     &           &                                                                        
		\end{tikzcd} \, .
		\]
		It has a three dimensional subrepresentation
		\[
		\begin{tikzcd}
			y\otimes y \arrow["q^{-2}"', loop, distance=3em, in=235, out=135, egreen] \arrow[r, "1", bend left, ered] & x\otimes y+q^{-1} y\otimes x \arrow[loop, "1"', distance=2em, in=-45, out=-135, egreen, looseness=8]   \arrow[l, "q+q^{-1}", bend left, eblue] \arrow[r, "q+q^{-1}", bend left, ered] & x\otimes x \arrow[l, "1", bend left, eblue] \arrow["q^{2}"', loop, distance=3em, in=45, out=-45, egreen]
		\end{tikzcd} \, . 
		\]
		The morphism $V\otimes V\to \mU$ is simply the quotient map with respect to this subrepresentation. One quickly computes that both sides of the relation are given by the morphism sending
		\[
			x\otimes \devil \otimes x \mapsto 0, \quad  x\otimes \devil \otimes y \mapsto q^{-1}, \quad y\otimes \devil \otimes x\mapsto 1, \quad y\otimes \devil \otimes y\mapsto 0.   \qedhere
		\]
	\end{proof}
	Further we will need one more easy lemma.
	\begin{lemm} \label{lemma without minus}
		We have an isomorphism 
		\[
		\purplecap\coloneqq \purplecaplabel \colon S\otimes S\to \mU,
		\]
		given on a basis by the assignment $\devil \otimes \devil \mapsto 1$. Furthermore the following Brauer relation between morphisms $S\otimes V\otimes S\to V$ holds:
		\[
		\cbox{
			\begin{tikzpicture}[tldiagram, scale=0.65]
				\posredbluecrossing{0}{1}
				\draw[epurple] \tlcoord{0}{0} \lineup \capright;
				\draw[black] \tlcoord{1}{2} \halflineup;
			\end{tikzpicture}
		} =  \cbox{
			\begin{tikzpicture}[tldiagram, scale=0.65]
				\posblueredcrossing{0}{0}
				\draw[epurple] \tlcoord{0}{2} \lineup \capleft;
				\draw[black] \tlcoord{1}{0} \halflineup;
			\end{tikzpicture}
		}.
		\]
	\end{lemm}
	\begin{proof}
		The isomorphism follows again from the action. One can check that both sides send $\devil \otimes x\otimes \devil \mapsto x$ and $\devil \otimes y\otimes \devil \mapsto -y$.
	\end{proof}
	Now we can prove devils sign theorem.
	\begin{proof}[Proof of Theorem \ref{theo no braiding}]
		We assume that a braiding $c$ exists. By Observation \ref{observation braiding candidate} we know there exists $\lambda\in k^{\times}$ such that $c_{S, V}=\lambda \cdot \iso$. To simplify reading the proof we will draw $\goldiso\coloneqq c_{S,V}$. 
		By the hexagon identities we have:
		\[
			\cbox{
			\begin{tikzpicture}[tldiagram, scale=0.65]
				\draw[black, dashed] \tlcoord{0}{1} \linewave{1}{-1};
				\draw[black, dashed] \tlcoord{0}{2} \linewave{1}{-1};
				\draw[white, double=epurple] \tlcoord{0}{0} \linewave{1}{2};
			\end{tikzpicture}
			}
			\coloneqq c_{S, V\otimes V}=(\id_V \otimes c_{S,V})\circ (c_{S,V}\otimes \id_V)=\cbox{
				\begin{tikzpicture}[tldiagram, scale=0.65]
					
					\draw[black, dashed] \tlcoord{0}{1} \linewave{1}{-1};
					\draw \tlcoord{1}{0} \lineup;
					\draw \tlcoord{0}{2} \lineup;
					\draw[black, dashed] \tlcoord{1}{2} \linewave{1}{-1};
					\draw[white, double=epurple] \tlcoord{0}{0} \linewave{1}{1} \linewave{1}{1};
				\end{tikzpicture}
			} .
		\]
		Further by the the fact that the braiding is a natural isomorphism $\blank\otimes \bblank \Rightarrow \bblank \otimes \blank$ and braiding with the monoidal unit is trivial, we have
		\begin{equation} \label{naturality of braiding}
			\cbox{
				\begin{tikzpicture}[tldiagram, scale=0.65]
					\draw[white, double=black, dashed] \tlcoord{0}{1} \linewave{1}{-1};
					\draw[white, double=black, dashed] \tlcoord{0}{2} \linewave{1}{-1};
					\draw \tlcoord{0.9}{0} \capright;
					\draw[white, double=epurple] \tlcoord{0}{0} \linewave{1}{2};
					\draw[epurple] \tlcoord{1}{2} \halflineup;
				\end{tikzpicture}
			} = \cbox{
			\begin{tikzpicture}[tldiagram, scale=0.65]
				\draw[black] \tlcoord{0}{1} \capright;
				\draw[white, double=epurple] \tlcoord{0}{0} \lineup;
			\end{tikzpicture}
			}.
		\end{equation}
		On the other hand we can compute using $\goldiso=\lambda \cdot \iso$ and see that
		\[
			\cbox{
				\begin{tikzpicture}[tldiagram, scale=0.65]
					\draw[white, double=black, dashed] \tlcoord{0}{1} \linewave{1}{-1};
					\draw[white, double=black, dashed] \tlcoord{0}{2} \linewave{1}{-1};
					\draw \tlcoord{0.9}{0} \capright;
					\draw[white, double=epurple] \tlcoord{0}{0} \linewave{1}{2};
					\draw[epurple] \tlcoord{1}{2} \halflineup;
				\end{tikzpicture}
			} =\lambda^2 \cdot \cbox{
			\begin{tikzpicture}[tldiagram, scale=0.65]
				\draw[white, double=black] \tlcoord{0}{1} \linewave{1}{-1} \capright;
				\draw[white, double=black] \tlcoord{0}{2} \linewave{1}{-1};
				\draw[white, double=epurple] \tlcoord{0}{0} \linewave{1}{2} \halflineup;
			\end{tikzpicture}
			}= \lambda^{2}\cdot \textcolor{epurple}{(-1)}\cdot \cbox{
			\begin{tikzpicture}[tldiagram, scale=0.65]
				\draw[white, double=black] \tlcoord{0}{2} \lineup \lineup \capleft \linewave{-1}{-1} \linewave{-1}{1};
				\draw[white, double=epurple] \tlcoord{0}{0} \linewave{1}{1} \linewave{1}{-1} \halflineup;
			\end{tikzpicture}
			}=(-\lambda^2)\cdot \cbox{
			\begin{tikzpicture}[tldiagram, scale=0.65]
				\draw[epurple] \tlcoord{0}{0} \lineup;
				\draw[white, double=black] \tlcoord{0}{1} \capright;
			\end{tikzpicture}
			}
		\]
		by Lemma \ref{lemma with minus}. Hence by naturality \ref{naturality of braiding}, we can can compare coefficients to conclude that $-\lambda^2=1$. This shows that $\lambda=\textcolor{epurple}{\pm i}\in \C$. Now we play the same game using $c_{S\otimes S, V}^{-1}\colon V\otimes S\otimes S\to S\otimes S\otimes V$.
		Similar to \ref{naturality of braiding} we should have
		\begin{equation} \label{equation needed}
		\cbox{
			\begin{tikzpicture}[tldiagram, scale=0.65]
				\draw[black, dashed] \tlcoord{0}{0} \linewave{1}{2};
				\draw \tlcoord{1}{2} \halflineup;
				\draw[white, double=epurple] \tlcoord{0}{1} \linewave{1}{-1};
				\draw[epurple] \tlcoord{1}{0} \capright;
				\draw[epurple] \tlcoord{0}{1} \linewave{1}{-1}; 
				\draw[white, double=epurple] \tlcoord{0}{2} \linewave{1}{-1};
			\end{tikzpicture}
		} = \cbox{
			\begin{tikzpicture}[tldiagram, scale=0.65]
				\draw \tlcoord{0}{0} \lineup;
				\draw[epurple] \tlcoord{0}{1} \capright;
			\end{tikzpicture}
		}.
		\end{equation}
		However using that $\goldinv=\frac{1}{\lambda} \cdot \inv$, our previous restraint on $\lambda$ and Lemma \ref{lemma without minus} we see that
		\[
		\cbox{
			\begin{tikzpicture}[tldiagram, scale=0.65]
				\draw[black, dashed] \tlcoord{0}{0} \linewave{1}{2};
				\draw \tlcoord{1}{2} \halflineup;
				\draw[white, double=epurple] \tlcoord{0}{1} \linewave{1}{-1};
				\draw[epurple] \tlcoord{1}{0} \capright;
				\draw[epurple] \tlcoord{0}{1} \linewave{1}{-1}; 
				\draw[white, double=epurple] \tlcoord{0}{2} \linewave{1}{-1};
			\end{tikzpicture}
		}=\textcolor{epurple}{(\mp i)}^2 \cbox{
		\begin{tikzpicture}[tldiagram, scale=0.65]
			\draw[black] \tlcoord{0}{0} \linewave{1}{2} \halflineup;
			\draw[white, double=epurple] \tlcoord{0}{1} \linewave{1}{-1};
			\draw[epurple] \tlcoord{1}{0} \capright;
			\draw[epurple] \tlcoord{0}{1} \linewave{1}{-1}; 
			\draw[white, double=epurple] \tlcoord{0}{2} \linewave{1}{-1};
		\end{tikzpicture}
		}=\textcolor{epurple}{(-1)}\cdot \cbox{
		\begin{tikzpicture}[tldiagram, scale=0.65]
			\draw[black] \tlcoord{0}{0} \linewave{1}{2} \halflineup;
			\draw[white, double=epurple] \tlcoord{0}{1} \linewave{1}{-1};
			\draw[epurple] \tlcoord{1}{0} \capright;
			\draw[epurple] \tlcoord{0}{1} \linewave{1}{-1}; 
			\draw[white, double=epurple] \tlcoord{0}{2} \linewave{1}{-1};
		\end{tikzpicture}
		}=\textcolor{epurple}{(-1)}\cdot \cbox{
		\begin{tikzpicture}[tldiagram, scale=0.65]
			\draw[black] \tlcoord{0}{0} \linewave{1}{1} \linewave{1}{-1} \halflineup;
			\draw[white, double=epurple] \tlcoord{0}{2} \lineup \lineup \capleft \linewave{-1}{-1} \linewave{-1}{1};
		\end{tikzpicture}
		} =\textcolor{epurple}{(-1)}\cdot \cbox{
		\begin{tikzpicture}[tldiagram, scale=0.65]
			\draw \tlcoord{0}{0} \lineup;
			\draw[epurple] \tlcoord{0}{1} \capright;
		\end{tikzpicture}
		},
		\]
		which is a contradiction to condition \ref{equation needed}. This finishes the proof.
	\end{proof}

\section{Braided module structure}
	\newcommand{\catOne}{\mathbf{Rep}^1}
	\newcommand{\catTwo}{\mathbf{Rep}^{-1}}
	\newcommand{\TL}{\mathbf{TL}}
	\newcommand{\hecke}{\mathcal{H}}
	Up to this point, the reader might be overwhelmed with the feeling of emptiness. To cheer them up, we will finish this paper by identifying a \emph{braided module category} structure on $\Rep$.
	We first describe two \textbf{braided} monoidal subcategories $\catOne,\catTwo$ of $\Rep$, which are the correct categories people are actually referring to when saying ``$\Rep$ is braided''.

	\subsection{Braided subcategories}
	\begin{defi}
		Let $\catOne$ be the smallest monoidal subcategory of $\Rep$ that is closed under direct sums and summands
		and contains the type $1$ standard representation $V$. 
		Similarly, let $\catTwo$ be the smallest monoidal subcategory of $\Rep$ that is closed under direct sums and summands and contains the type $-1$ standard representation $V\otimes S$. 
	\end{defi}
	
	Both $\catOne,\catTwo$ are Karoubian closures of the Temperley--Lieb category $\TL(\delta)$ for an appropriate choice of parameter $\delta$: 
	\[
		\catOne \cong \operatorname{Kar}(\TL(-[2]) ),\qquad \catTwo \cong \operatorname{Kar}(\TL([2]) )
	\]
	Neither of this categories contains the type $-1$ trivial representation $S$, which led to problematic sign $\textcolor{epurple}{-1}$. 

	These braidings on $\catOne,\catTwo $ both arise from the Hecke algebra action.
	Let $\hecke_d$ be the Hecke algebra of the symmetric group $S_d$. It is  generated by $H_1,\ldots,H_{d-1}$ subject to the following relations.
	\begin{align*}
		\begin{gathered}
		H_i^2 = 1+(q^{-1}-q)H_i\;
			,\;
		H_i H_{i+1} H_i = H_{i+1} H_i H_{i+1}
			,\; 
		H_i H_j = H_j H_i \text{ for } |i-j|>1.
		\end{gathered}
	\end{align*}
	The second and the third relations are the braid group relations, hence $\hecke_d$ is a quotient of the group algebra of the braid group on $d$ strands by the first quadratic relation. 
	Under the basis $x\otimes x ,y \otimes x, x\otimes y, y\otimes y$,
	the following matrices on $V\otimes V$ define an action of $\hecke_d$ on $V^{\otimes d}$ respectively $(S\otimes V)^{\otimes d}$ which factor through the Temperley--Lieb algebra.
	\small
	\begin{align}
	R^{-1}=
	\left( \begin{matrix}
		-q &&&\\
		&q^{-1}-q &1  \\
		& 1&& \\
		&&& -q 
	\end{matrix}\right) , \quad\quad
	R^{1}
	=\left( \begin{matrix}
		q^{-1}  &&&\\
		&&1&  \\
		& 1&q^{-1} -q  & \\
		&&& q^{-1}  
	\end{matrix}\right) .
	\label{Rmat}
	\end{align}
	\normalsize
	Under our conventions, $R^{-1}$ respectively $R^{1}$ is given by the following morphisms
	\[
		 \cbox{
			\begin{tikzpicture}[tldiagram, scale=0.65]
				\draw[black] \tlcoord{-1}{0}  \capleft;
				\draw[black] \tlcoord{0}{0} \cupleft;
			\end{tikzpicture}
		} 
		-q  \cbox{
			\begin{tikzpicture}[tldiagram, scale=0.65]
				\draw[black] \tlcoord{0}{0} \lineup;
				\draw[black] \tlcoord{0}{1} \lineup;
			\end{tikzpicture}
		}
		\in \TL(-[2])
		,\qquad
		 \cbox{
			\begin{tikzpicture}[tldiagram, scale=0.65]
				\draw[black] \tlcoord{-1}{0}  \capleft;
				\draw[black] \tlcoord{0}{0} \cupleft;
			\end{tikzpicture}
		} 
		+q^{-1}  \cbox{
			\begin{tikzpicture}[tldiagram, scale=0.65]
				\draw[black] \tlcoord{0}{0} \lineup;
				\draw[black] \tlcoord{0}{1} \lineup;
			\end{tikzpicture}
		}
		\in \TL([2]).
	\]
	Here the cup diagrams are the Temperley-Lieb generators instead of the diagrammatics used in \S1.
	We refer to \cite{St-ICM-braid} for more details about the significance of type $-1$ representations, which also motivated the idea of considering simultaneously type $1$ and type $-1$ representations.

	\subsection{Braided module categories}
	There are various different notions of braided module structures, many of which differ from each other by some $\mathbb{Z}$-twisting, see \cite[Rem. 3.6]{BeZvBrJo-quantum-character-braided} for details. Here we use the same definition as \cite{Ko-braided-module}.
	
	\begin{defi}
		Let  $\mathbf{C}$ be a monoidal category. A (right) \emph{module category} over $\mathbf{C}$ is the data of a pair  $(\mathbf{M}, \triangleleft)$ of a category $\mathbf{M}$ together with a functor $\triangleleft : \mathbf{M} \times \mathbf{C} \to \mathbf{M}$.  
	\end{defi}

	\newcommand{\brmd}{\eta}
	\newcommand{\ftwist}[1][d]{C_{#1}}
	\newcommand{\SleftCat}{\mathbf{SV}}

	Clearly, for any monoidal category $\mathbf{C}$ and any monoidal subcategory $\mathbf{D}\subset \mathbf{C}$, the restriction of the tensor product to $\mathbf{C}\times \mathbf{D} \to  \mathbf{C}$ canonically equips $ \mathbf{C}$  with a $\mathbf{D}$-module category structure. In particular, $\Rep$ is a module category over both $\catOne$ and $\catTwo$.

	\begin{defi}
		Let $(\mathbf{C}, R)$ be a braided monoidal category and $(\mathbf{M}, \triangleleft)$ a module category. A \emph{braided module category} structure is a natural transformation $\brmd: \triangleleft \to \triangleleft\,$, such that for any objects $M \in \mathbf{M}$, $A,B\in \mathbf{C}$ hold the following:
		\begin{align}
			&\brmd_{M, A\otimes B} =  
			(\id_M \triangleleft R_{B,A})
			\circ	( \brmd_{M, B}\triangleleft \id_A ) 
			\circ  (\id_M \triangleleft R_{A,B})
			\circ  ( \brmd_{M, A}\triangleleft \id_B )  
			, \label{braided-mod-1}
			\\
			&\brmd_{M \triangleleft A,B} =   (\id_M \triangleleft R_{B,A})\circ ( \brmd_{M, B}\triangleleft \id_A ) \circ  (\id_M \triangleleft R_{A,B}) 
			.\label{braided-mod-2}
		\end{align}
	\end{defi}

	\tikzset{anchorbase/.style={baseline={([yshift=-0.5ex]current bounding box.center)}}}

	Recall that $\hecke_d$ deforms the group algebra of $S_d$, and for each element $w\in S_d$ there is a corresponding standard basis element $H_w$ via the Matsumoto section. Let $w_0 \in S_d$ be the longest element. Let $\ftwist = H_{w_0}^2\in  \hecke_d$ be the full twist. It is know that $\ftwist$ lies in the center of $\hecke_d$, see \cite[{Thm.4.2}]{Gonz-braid}. Explicitly, we have \[
		\ftwist = (H_{n-1} \ldots  H_2 H_1 )\cdot (H_{n-1} \ldots  H_2 )\cdot \cdots\cdot (H_{n-1}H_{n-2}) \cdot H_{n-1}.
	\]
	\begin{theo}
		There exists a unique braided $\catOne$-module structure $\brmd$ on $\Rep$ such that \begin{align}
			\brmd_{S, V^{\otimes d}} = \id_{S}\otimes \ftwist ,\qquad \brmd_{\mU, V^{\otimes d}} =\id_{\mU}\otimes \ftwist
			.\label{braide-mod-structure}
		\end{align} 
	\end{theo}
	\begin{proof}
		Let $\SleftCat \subset \Rep$ be the full subcategory given by objects of the form $V^{\otimes d}, S\otimes V^{\otimes d}$. By definition we have $\SleftCat\otimes \catOne \subset \SleftCat$, i.e. $\SleftCat$ is also a module category over $\catOne$.
		We show that \ref{braide-mod-structure} defines a braided module structure on $\SleftCat$. Note that the second axiom \ref{braided-mod-2} forces \begin{align}
			\brmd_{S\otimes V^{\otimes a}, V^{\otimes d}} = (\id_S\otimes  R_{d,a} )\circ \ftwist[d] \circ  R_{a,d}
			,\quad 
			\brmd_{V^{\otimes a}, V^{\otimes d}} = R_{d,a} \circ \ftwist[d] \circ  R_{a,d},
			\label{BM2-for-Rep}
		\end{align}
		where the braiding $R$ on $\catOne$ is given by \[
			R_{a,b} = \begin{tikzpicture}[
				strand/.style={line width=.9pt,line cap=round},
				over/.style={strand,preaction={draw=white,line width=3.5pt}},
				coupon/.style={draw,fill=white,line width=.9pt,
					minimum width=1.3cm,minimum height=.51cm,
					inner sep=0pt},
				bundle brace/.style={decorate,
					decoration={brace,amplitude=4pt},line width=.7pt}
				,
				baseline={([yshift=-1.5ex]current bounding box.center)},
				every path/.style = {out=90,in=-90, thick}
				,yscale = 0.5
				]
				\def\dist{1}

				\draw (2*\dist,2) to (0,4);
				\draw (3*\dist,2) to (\dist,4);

				\draw[over] (0,2) to (2*\dist,4) ;
				\draw[over] (\dist,2) to (3*\dist,4);
				\node at (.6,4) {$\cdots$};
				\node at (2.5,4) {$\cdots$};

				\draw[bundle brace] (-.08,4.2) -- (1.08,4.20)
					node[midway,above=8pt] {$a\ \mathrm{strands}$};
				\draw[bundle brace] (1.92,4.20) -- (3.08,4.20)
					node[midway,above=8pt] {$b\ \mathrm{strands}$};
			\end{tikzpicture}.
		\]
		By the naturality of $R$ and the fact that $\ftwist[d]$ lies in the center of $\hecke_d$, $\brmd_{-,-}$ is natural in both components.
		To check the first axiom \ref{braided-mod-1}, we first claim that 
		\begin{align}	
			\ftwist[a+b] =
			\begin{tikzpicture}[
				strand/.style={line width=.9pt,line cap=round},
				over/.style={strand,preaction={draw=white,line width=3.5pt}},
				coupon/.style={draw,fill=white,line width=.9pt,
					minimum width=1.3cm,minimum height=.51cm,
					inner sep=0pt},
				bundle brace/.style={decorate,
					decoration={brace,mirror,amplitude=4pt},line width=.7pt}
				,anchorbase,
				every path/.style = {out=90,in=-90, thick}
				,yscale = 0.5
				]
				\def\dist{1}
				\draw (2*\dist,0) to (0,2);
				\draw (3*\dist,0) to (\dist,2);
				\draw[over] (0,0) to (2*\dist,2);
				\draw[over] (\dist,0) to (3*\dist,2);
				\draw (2*\dist,2) to (0,4)to (0,5);
				\draw (3*\dist,2) to (\dist,4)to (\dist,5);
				\draw[over] (0,2) to (2*\dist,4) to (2*\dist,5);
				\draw[over] (\dist,2) to (3*\dist,4)to (3*\dist,5);
				\node[coupon] at (.54*\dist,4) {$\ftwist[a]$};
				\node[coupon] at (.54*\dist,2) {$\ftwist[b]$};
				\node at (.5,0) {$\cdots$};
				\node at (2.5,0) {$\cdots$};
				\node at (.6,5) {$\cdots$};
				\node at (2.5,5) {$\cdots$};
				\draw[bundle brace] (-.08,-.20) -- (1.08,-.20)
					node[midway,below=8pt] {$a\ \mathrm{strands}$};
				\draw[bundle brace] (1.92,-.20) -- (3.08,-.20)
					node[midway,below=8pt] {$b\ \mathrm{strands}$};
			\end{tikzpicture}
			=
			\begin{tikzpicture}[
				strand/.style={line width=.9pt,line cap=round},
				over/.style={strand,preaction={draw=white,line width=3.5pt}},
				coupon/.style={draw,fill=white,line width=.9pt,
					minimum width=1.3cm,minimum height=.51cm,
					inner sep=0pt},
				bundle brace/.style={decorate,
					decoration={brace,mirror,amplitude=4pt},line width=.7pt}
				,anchorbase,
				every path/.style = {out=90,in=-90, thick}
				,yscale = 0.5
				]
				\def\dist{1}
				\draw (2*\dist,0) to (0,2);
				\draw (3*\dist,0) to (\dist,2);
				\draw[over] (0,0) to (2*\dist,2);
				\draw[over] (\dist,0) to (3*\dist,2);
				\draw (2*\dist,2) to (0,4);
				\draw (3*\dist,2) to (\dist,4);
				\draw[over] (0,2) to (2*\dist,4) ;
				\draw[over] (\dist,2) to (3*\dist,4);
				\node[coupon] at (2.54*\dist,2) {$\ftwist[a]$};
				\node[coupon] at (.54*\dist,2) {$\ftwist[b]$};
				\node at (.5,0) {$\cdots$};
				\node at (2.5,0) {$\cdots$};
				\node at (.6,4) {$\cdots$};
				\node at (2.5,4) {$\cdots$};
				\draw[bundle brace] (-.08,-.20) -- (1.08,-.20)
					node[midway,below=8pt] {$a\ \mathrm{strands}$};
				\draw[bundle brace] (1.92,-.20) -- (3.08,-.20)
					node[midway,below=8pt] {$b\ \mathrm{strands}$};
			\end{tikzpicture}.
		\label{full-twist-identity}
		\end{align}
		Since by definition we have $C_a= H_{w_a}^2$,
		writing $w_n$ for the longest word of $S_n$, this follows immediately from the fact that \[
			H_{w_{a+b}} =  
			\begin{tikzpicture}[
				strand/.style={line width=.9pt,line cap=round},
				over/.style={strand,preaction={draw=white,line width=3.5pt}},
				coupon/.style={draw,fill=white,line width=.9pt,
					minimum width=1.3cm,minimum height=.51cm,
					inner sep=0pt},
				bundle brace/.style={decorate,
					decoration={brace,mirror,amplitude=4pt},line width=.7pt}
				,anchorbase,
				every path/.style = {out=90,in=-90, thick}
				,yscale = 0.5
				]
				\def\dist{1}
				\draw (2*\dist,0) to (0,2);
				\draw (3*\dist,0) to (\dist,2);

				\draw[over] (0,0) to (2*\dist,2);
				\draw[over] (\dist,0) to (3*\dist,2);

				\node[coupon] at (2.54*\dist,2) {$w_a$};
				\node[coupon] at (.54*\dist,2) {$w_b$};
				\node at (.5,0) {$\cdots$};
				\node at (2.5,0) {$\cdots$};

				\draw[bundle brace] (-.08,-.20) -- (1.08,-.20)
					node[midway,below=8pt] {$a\ \mathrm{strands}$};
				\draw[bundle brace] (1.92,-.20) -- (3.08,-.20)
					node[midway,below=8pt] {$b\ \mathrm{strands}$};
			\end{tikzpicture}
			=
			\begin{tikzpicture}[
				strand/.style={line width=.9pt,line cap=round},
				over/.style={strand,preaction={draw=white,line width=3.5pt}},
				coupon/.style={draw,fill=white,line width=.9pt,
					minimum width=1.3cm,minimum height=.51cm,
					inner sep=0pt},
				bundle brace/.style={decorate,
					decoration={brace,amplitude=4pt},line width=.7pt}
				,anchorbase,
				every path/.style = {out=90,in=-90, thick}
				,yscale = 0.5
				]
				\def\dist{1}

				\draw (2*\dist,2) to (0,4);
				\draw (3*\dist,2) to (\dist,4);

				\draw[over] (0,2) to (2*\dist,4) ;
				\draw[over] (\dist,2) to (3*\dist,4);

				\node[coupon] at (2.54*\dist,2) {$w_a$};
				\node[coupon] at (.54*\dist,2) {$w_b$};
				\node at (.6,4) {$\cdots$};
				\node at (2.5,4) {$\cdots$};

				\draw[bundle brace] (-.08,4.2) -- (1.08,4.20)
					node[midway,above=8pt] {$a\ \mathrm{strands}$};
				\draw[bundle brace] (1.92,4.20) -- (3.08,4.20)
					node[midway,above=8pt] {$b\ \mathrm{strands}$};
			\end{tikzpicture}.
		\]
		To see that $H_{w_{a+b}}$ is indeed given by these pictures, just note that both pictures are reduced expresssions, and the longest word is a reflection that reverses the indices. 
		Now the axiom \ref{braided-mod-1} amounts to \[
			\begin{tikzpicture}[
				strand/.style={line width=.9pt,line cap=round},
				over/.style={strand,preaction={draw=white,line width=3.5pt}},
				coupon/.style={draw,fill=white,line width=.9pt,
					minimum width=1.3cm,minimum height=.51cm,
					inner sep=0pt},
				bundle brace/.style={decorate,
					decoration={brace,mirror,amplitude=4pt},line width=.7pt}
				,anchorbase,
				every path/.style = {out=90,in=-90, thick}
				,yscale = 0.5
				]
				\def\dist{1}
				\draw (2*\dist,0) to (0,2);
				\draw (3*\dist,0) to (\dist,2);
				\draw[over] (0,0) to (2*\dist,2);
				\draw[over] (\dist,0) to (3*\dist,2);
				\draw (2*\dist,2) to (0,4);
				\draw (3*\dist,2) to (\dist,4);
				\draw[over] (0,2) to (2*\dist,4) ;
				\draw[over] (\dist,2) to (3*\dist,4);
				\node[coupon] at (.54*\dist,2) {$\ftwist[{a+b}]$};
				\node at (.5,0) {$\cdots$};
				\node at (2.5,0) {$\cdots$};
				\node at (.6,4) {$\cdots$};
				\node at (2.5,4) {$\cdots$};
				\draw[bundle brace] (-.08,-.20) -- (1.08,-.20)
					node[midway,below=8pt] {$m\ \mathrm{strands}$};
				\draw[bundle brace] (1.92,-.20) -- (3.08,-.20)
					node[midway,below=8pt] {$a+b\ \mathrm{strands}$};
			\end{tikzpicture}
			=
			\begin{tikzpicture}[
				strand/.style={line width=.9pt,line cap=round},
				over/.style={strand,preaction={draw=white,line width=3.5pt}},
				coupon/.style={draw,fill=white,line width=.9pt,
					minimum width=1.13cm,minimum height=.51cm,
					inner sep=0pt},
				bundle brace/.style={decorate,
					decoration={brace,mirror,amplitude=4pt},line width=.7pt}
				,anchorbase,
				every path/.style = {in=-90,out=90, thick}
				,yscale = 0.45
				]
				\def\dist{0.75}
				\draw[blue,in=-90,out=150] (4*\dist,0) to (0*\dist,2);
				\draw[blue,in=-90,out=90] (5*\dist,0) to (1*\dist,2);
				\draw[over,] (0,0) to (2*\dist,2);
				\draw[over,] (\dist,0) to (3*\dist,2);
				\draw[over,red] (2*\dist,0) to (4*\dist,2);
				\draw[over,red] (3*\dist,0) to (5*\dist,2);
				\draw[] (2*\dist,2) to (0,4);
				\draw[] (3*\dist,2) to (\dist,4);
				\draw[over,red]  (4*\dist,2) to (0*\dist,6);
				\draw[over,red] (5*\dist,2)to (1*\dist,6);
				\draw[over,] (0,4)to (2*\dist,6);
				\draw[over,] (\dist,4)to (3*\dist,6);
				\draw[over,blue] (0,2) to (4*\dist,5)  to (4*\dist,8);
				\draw[over,out=30,blue] (\dist,2.5) to (5*\dist,5) [out=90]to (5*\dist,8) ;
				\draw[] (2*\dist,6) to (0,8);
				\draw[] (3*\dist,6) to (\dist,8);
				\draw[over,red] (0,6)to (2*\dist,8);
				\draw[over,red] (\dist,6)to (3*\dist,8);
				\node[coupon] at (.54*\dist,6) {$\ftwist[a]$};
				\node[coupon] at (.54*\dist,2) {$\ftwist[b]$};
				\node at (.5*\dist,0) {$\cdots$};
				\node at (2.5*\dist,0) {$\cdots$};
				\node at (4.5*\dist,0) {$\cdots$};
				\node at (.6*\dist,8) {$\cdots$};
				\node at (2.5*\dist,8) {$\cdots$};
				\node at (4.5*\dist,8) {$\cdots$};
				\draw[bundle brace,] (-.08*\dist,-.20) -- (1.08*\dist,-.20)
					node[midway,below=8pt] {$m\ \mathrm{strands}\;\;$};
				\draw[bundle brace,red] (1.92*\dist,-.20) -- (3.08*\dist,-.20)
					node[midway,below=8pt] {$a\ \mathrm{strands}$};
				\draw[bundle brace,blue] (3.92*\dist,-.20) -- (5.08*\dist,-.20)
					node[midway,below=8pt] {$\;\;b\ \mathrm{strands}$};
			\end{tikzpicture}
		\]
		This follows from applying \ref{full-twist-identity} to the colored strands in the following equality.\[
			RHS= \begin{tikzpicture}[
				strand/.style={line width=.9pt,line cap=round},
				over/.style={strand,preaction={draw=white,line width=3.5pt}},
				coupon/.style={draw,fill=white,line width=.9pt,
					minimum width=1.13cm,minimum height=.51cm,
					inner sep=0pt},
				bundle brace/.style={decorate,
					decoration={brace,mirror,amplitude=4pt},line width=.7pt}
				,anchorbase,
				every path/.style = {in=-90,out=90, thick}
				,yscale = 0.45
				]
				\def\dist{0.75}
				\draw[blue,in=-90,out=150] (4*\dist,0) to (0*\dist,2);
				\draw[blue,in=-90,out=90] (5*\dist,0) to (1*\dist,2);
				\draw[over,red] (2*\dist,0) to (4*\dist,2);
				\draw[over,red] (3*\dist,0) to (5*\dist,2);
				\draw[over,red]  (4*\dist,2) to (0*\dist,6);
				\draw[over,red] (5*\dist,2)to (1*\dist,6);
				\draw[over,out=80] (0,0) to (4.7*\dist,3.8);
				\draw[over,out=60] (\dist,0) to (5.7*\dist,3.8);
				\draw[over,] (4.7*\dist,3.8) to (0*\dist,8);
				\draw[over,] (5.7*\dist,3.8)to (1*\dist,8);
				\draw[over,blue] (0,2) to (4*\dist,5)  to (4*\dist,8);
				\draw[over,out=30,blue] (\dist,2.5) to (5*\dist,5) [out=90]to (5*\dist,8) ;
				\draw[over,red] (0,6)to (2*\dist,8);
				\draw[over,red] (\dist,6)to (3*\dist,8);
				\node[coupon] at (.56*\dist,5.6) {$\ftwist[a]$};
				\node[coupon] at (.54*\dist,2.5) {$\ftwist[b]$};
				\node at (.5*\dist,0) {$\cdots$};
				\node at (2.5*\dist,0) {$\cdots$};
				\node at (4.5*\dist,0) {$\cdots$};
				\node at (.6*\dist,8) {$\cdots$};
				\node at (2.5*\dist,8) {$\cdots$};
				\node at (4.5*\dist,8) {$\cdots$};
				\draw[bundle brace,] (-.08*\dist,-.20) -- (1.08*\dist,-.20)
					node[midway,below=8pt] {$m\ \mathrm{strands}\;\;$};
				\draw[bundle brace,red] (1.92*\dist,-.20) -- (3.08*\dist,-.20)
					node[midway,below=8pt] {$a\ \mathrm{strands}$};
				\draw[bundle brace,blue] (3.92*\dist,-.20) -- (5.08*\dist,-.20)
					node[midway,below=8pt] {$\;\;b\ \mathrm{strands}$};
			\end{tikzpicture}
		\]
		This proves that \ref{BM2-for-Rep} defines a braided $\catOne$-module category structure on $\Rep$.
		
		Finally note that $\SleftCat$ is equivalent to $\Rep$ as ordinary categories, because the natural embedding $\iota: \SleftCat \hookrightarrow \Rep$ is fully faithful, and by \ref{lemma V otimes S}, this is also essentially surjective. 
		Clearly, $\iota$ is also an equivalence of $\catOne$-module categories.
		Let $\Gamma: \Rep \to \SleftCat$ be a quasi-inverse of $\iota$ compatible with the $\catOne$-module category structure.
		Thus, we have $f:  \id_{\Rep}\cong \iota\circ \phi $. \[
			\begin{tikzcd}
				\SleftCat \arrow[r, "\iota",bend left= 20] & \Rep \arrow[l,"\Gamma", bend left=20]
			\end{tikzcd}
		\]
		Next, let $\triangleleft: \Rep\times \catOne \to \Rep$ be the action functor given by the tensor product. Let $\trianglelefteq = \triangleleft \circ \iota$ be the restriction of $\triangleleft $ to $\SleftCat$. We define a natural transformation $\brmd' : \trianglelefteq\to \trianglelefteq$ by setting \[
			\brmd'_{S\otimes V^{\otimes a}, V^{\otimes b}} = \id_S \otimes \ftwist[a+b]
			,\qquad
			\brmd'_{ V^{\otimes a}, V^{\otimes b}} =  \ftwist[a+b]
		\]
		Indeed, since $\ftwist[a+b]$ is central in $\hecke_{a+b}$, our $\brmd'$ is natural in both components.
		We then define a natural transformation \[
			\brmd : \begin{tikzcd}
				\triangleleft \arrow[r,"\id_\triangleleft \circ f", "\sim"swap] &
				\trianglelefteq  \circ \Gamma\arrow[r,"\brmd'\circ \id_\Gamma"]& 
				\trianglelefteq \circ \Gamma \arrow[r,"\id_\triangleleft \circ f^{-1} "]&
				\triangleleft
			\end{tikzcd} .
		\]
		Since everything is compatible with the $\catOne$-module structure, we conclude that $\brmd$ is the desired braided module structure on $\Rep$. 
	\end{proof}

	\begin{rema}
		Note that the braided module category construction in this section does not use
		the fact that we have $U_q(\mathfrak{sl}_2)$-representations. In fact, the same argument proves that the category of all finite dimensional $U_q(\mathfrak{gl}_n)$-modules is a braided module category over its monoidal subcategories generated by the corresponding type I or type II natural representations. The braided monoidal structure on these subcategories are given again by \ref{Rmat}.
	\end{rema}

	The full twists give rise to a Serre functor of a certain category, see \cite{Go-Ho-Me-Serre-functor-link-homology}. We do not understand the meaning of the braided module structure given by the full twist and would be happy to learn about it.

	\printbibliography
	
\end{document}